\documentclass [a4paper,11pt,oneside ]{amsproc}
\usepackage{graphicx}
\usepackage{amsfonts}
\usepackage{amsmath}
\usepackage{parskip}
\usepackage{color}

\RequirePackage{srcltx}

\numberwithin{equation}{section}
\renewcommand\r{\right\rangle}
\renewcommand\l{\left\langle}

\newcommand\dsize{\displaystyle}
\newcommand\supp{\operatorname{supp}}

\newcommand\cal{\mathcal}

\newcommand \C{\mathbb{C}}
\newcommand\Q{\mathbb{Q}}
\newcommand\Z{\mathbb{Z}}
\newcommand\N{\mathbb{N}}
\newcommand\F{{\bf F}}

\newcommand\T{{ \bf T}}
\newcommand\G{{\mathcal{G}}}

 \newcommand\M{{\bf M}}

\newtheorem{Thm}{Theorem}[section]
\newtheorem{Lemma}[Thm]{Lemma}
\newtheorem{Cor}[Thm]{Corollary}
\newtheorem{Prop}[Thm]{Proposition}
\theoremstyle{remark}
\newtheorem{Rem}{Remark}[section]

 \newcommand {\rank}{{\rm Rank}\,}
 \newcommand {\diag}{{\rm Diag}\,}

\begin{document}
	
	\title[]{   Block-equivalent finite Gabor frames  }
	
	\author{Oleg Asipchuk }
	\address{O.~Asipchuk, University of Cincinnati, 
		Department of Mathematics,
		Cincinnati, OH 45221, USA}
	
	\author{Laura De Carli}
	\address{L.~De Carli, Florida International University,
		Department of Mathematics,
		Miami, FL 33199, USA}
	%
	\author{Luis Rodriguez}
	\address{L.~Rodriguez, Florida International University,
		Department of Mathematics,
		Miami, FL 33199, USA}
	

 \begin{abstract}
We study finite systems of vectors whose frame operator matrices are unitarily equivalent, via explicit and computationally efficient unitary transformations, to block-diagonal matrices. We call such systems {\it block-equivalent}. 

We show that a Gabor system $\G=\G(g,L\times K)\subset \mathbb C^N$ is block-equivalent when  either the modulation set $L$ or the translation set $K$ is a subgroup of $\mathbb Z_N$. We also characterize situations in which the frame operator matrix becomes diagonal.

Finally, we show that geometric conditions on subsets of $\mathbb Z_N$ force certain diagonals of the frame operator matrix of $\G$ to vanish, yielding additional sparsity and block structures.
	  
\end{abstract}

 \maketitle

\section{Introduction}	

In finite dimensions, reconstructing a vector from its frame coefficients typically relies on the inverse of the frame operator. Given a frame $\mathcal{F}=\{f_i\}_{i\in J}$ with frame operator matrix $S$, the {\it canonical dual frame} $\{g_i=S^{-1} f_i\}_{i\in J}$, defined using the inverse of the frame operator matrix of $\mathcal{F}$, satisfies the reconstruction formula:

\begin{equation}\label{rec-formulas}
	x=\sum_{i\in J} \langle x,f_i\rangle g_i \;=\; \sum_{i\in J} \langle x,g_i\rangle f_i,
\end{equation}
for all $x \in \C^N$, a property shared by all dual frames. Computing $S^{-1}$ can be computationally challenging, especially in high-dimensional spaces.

Tight frames are often considered ideal because their frame operator is a scalar multiple of the identity making inversion trivial. This structure allows for reconstruction formulas analogous to those of orthonormal bases while retaining the redundancy and robustness that make frames valuable in applications. Unfortunately, tight frames rarely arise in practical applications, and even scalable frames (those that can be made tight by appropriately rescaling the frame vectors) are difficult to construct.

The main advantage of a tight frame is the easy inversion of its frame operator; however, also frames whose  frame operator matrix  is diagonal or block-diagonal can easily be inverted.

In this paper, we study finite systems of vectors (not just frames) whose frame operator matrices can be transformed into block-diagonal matrices using explicit and computationally efficient unitary transformations. We introduce the following definitions:

A finite set of vectors ${\cal F}\subset \C^N$ is {\it block-equivalent} if its frame operator matrix is unitarily similar to a block-diagonal matrix via an explicit,  inexpensive to apply, unitary transformation (such as a permutation matrix or a generalized discrete Fourier transform matrix). Two matrices $A$ and $B$ are unitarily similar if there exists a unitary matrix $U$ such that $B=U^*AU$, where $U^*$ is the conjugate transpose of $U$.

If this explicit transformation yields a diagonal matrix, we call the set  {\it diagonal-equivalent}; additionally, if the frame operator matrix is   block-diagonal (respectively, {\it diagonal}), we refer to it as a {\it block-diagonal set} (respectively, {\it diagonal set}).

Because the frame operator is self-adjoint, it is unitarily similar to a diagonal matrix.  However, computing a diagonalizing transformation generally requires finding eigenvalues and eigenvectors of the frame operator, which is computationally expensive. True computational efficiency is achieved only when the transformation is explicitly known and easy to apply.

Frames that are unitarily equivalent to a block-diagonal matrix via an explicit transformation can be used to produce  efficient reconstruction formulas. To see   how this is achieved, let $S$ be the frame operator matrix of a frame  ${\cal F} = \{f_i\}_{i\in J}$. If $S = U^*BU$, where $U$ is a known unitary matrix and $B$ is a block-diagonal matrix, then $S^{-1} = U^* B^{-1}U$. Substituting this into the reconstruction formula \eqref{rec-formulas} yields $$x = \sum_{i\in J} \langle x, f_i\rangle U^* B^{-1} U f_i.$$ Multiplying by $U$ and utilizing the unitary property $\langle x, f_i \rangle = \langle Ux, Uf_i \rangle$, we obtain an efficient reconstruction formula for the transformed vector $y = Ux$:
$$ y = \sum_{i\in J} \langle y, U f_i\rangle B^{-1} (U f_i). $$
$B^{-1}$ can be computed efficiently by simply inverting its individual blocks. Thus, we can efficiently reconstruct $y$ using the transformed frame $\{Uf_i\}_{i\in J}$, and subsequently recover $x = U^*y$.

In this paper  we focus on finite Gabor systems and investigate the specific conditions, such as lattice structure and window vector construction, that make them block-equivalent or diagonal-equivalent. Finite Gabor frames represent vectors in $\C^N$ as linear combinations of time-frequency shifts of a window vector $g \in \C^N$, and play a central role in finite-dimensional time-frequency analysis.

Finite Gabor frames may be viewed as finite-dimensional discretizations of classical Gabor systems on $L^2(\mathbb R)$, obtained by replacing continuous translations and modulations with cyclic time-frequency shifts on $\mathbb Z_N$  \cite{Orr, Jansen, Sonder}. See also the excellent \cite{S-thesis},    the recent \cite{5auth}  and the references cited there.

Let $\zeta_N = e^{2\pi i/N}$ denote the $N$-th primitive root of unity. Given $k, \ell \in \Z_N$, we define the cyclic translation operator $T_k$ and the modulation operator $M_\ell$ by their action on a vector $v \in \C^N$:
\begin{equation}\label{e-transl-mod}(T_k v)[j] = v[j-k] \quad \text{and} \quad (M_\ell v)[j] = \zeta_N^{\ell j} v[j] \end{equation}
for $j = 0, 1, \dots, N-1$. By cyclic translation, we mean that the indices are taken modulo $N$. For a non-zero window vector $g \in \C^N$ and a time-frequency set $\Lambda \subset \Z_N \times \Z_N$, a finite Gabor system is the set of time-frequency shifts of $g$:
\begin{equation}\label{e-Gabor-sys} \G(g, \Lambda) = \{M_\ell T_k g : (\ell, k) \in \Lambda\}. \end{equation}

When the time-frequency set is a Cartesian product $\Lambda = L \times K$ for subsets $L, K \subset \Z_N$, we refer to $L$ as the set of modulations and $K$ as the set of translations.

In general, verifying whether a given finite Gabor system constitutes a frame is a nontrivial problem \cite{ Malikiosis, PfanderFinite, Qiu}.

The structure of Gabor frame operators and its use in efficient
reconstruction have been extensively studied. For lattice Gabor systems,
Janssen's representation provides a fundamental description of the frame
operator \cite{Jan}. Unitary factorizations and discrete Zak-transform
methods lead to efficient computational schemes
\cite{Strom,BastiaansGeilen,Red,Son}.
A complementary approach is provided by painless nonorthogonal
expansions, introduced by Daubechies, Grossmann, and Meyer \cite{Dau}
and extended to nonstationary Gabor systems, including finite-dimensional
settings, by Balazs et al.\ \cite{Bal2}. In the painless case, the
frame operator is diagonal (a multiplication operator in the continuous
setting), and the canonical dual is explicit.
Iterative methods for computing canonical dual and tight windows
are developed by Janssen and S{\o}ndergaard \cite{Son2}.

Our first result gives an explicit block decomposition for finite Gabor
systems whose time-frequency set is a Cartesian product $L\times K$,
assuming that at least one of the two factors is a subgroup of
$\mathbb Z_N$. The other factor may be an arbitrary subset, and the
window vector is unrestricted. When both factors are subgroups, the
decomposition can also be obtained from Janssen's representation of the
frame operator \cite{Jan}; related unitary factorizations in the finite
setting are developed in \cite{Strom}.
Our formulation treats the two cases within a common framework and
specifies transformations that depend only on the subgroup structure,
without requiring an eigenvalue decomposition of the frame operator.
%
We summarize this structural result in the following
theorem.

\begin{Thm} \label{T-Main}
  Let \(\mathcal G=\mathcal G(g,L\times K)\) be a finite Gabor system in \(\mathbb C^N\).If either the modulation set $L$ or the translation set $K$ is a subgroup of $\Z_N$, then $\G$ is block-equivalent. Specifically:
  \begin{itemize}
\item  If \(L\) is a subgroup of order \(r\), its frame operator is  unitarily similar, through a permutation,  to a block-diagonal matrix with \(r\) blocks of size \(N/r\).
\item  If \(K\) is a subgroup of order \(p\), its frame operator is unitarily similar, through a discrete Fourier transform and a permutation, to a block-diagonal matrix with \(p\) blocks of size \(N/p\).\end{itemize}
\end{Thm}

The explicit structure of the frame operator matrix in these cases, including the precise composition of the blocks and the exact unitary transformations required to achieve them, is formalized as a corollary in Section 5.

When either the set of modulations or the set of translations consists of the entire group $\Z_N$, the Gabor system becomes diagonal or diagonal-equivalent:

\begin{Cor}\label{full-sampling}
	Let $\G = \G(g, L \times K)$ be a finite Gabor system in $\mathbb{C}^N$.
	
	a) (Full set of modulations) If $L = \Z_N$, then  the frame operator matrix of $\G$ is   diagonal. Furthermore, $\G$ is a frame for $\C^N$ if and only if $\dsize\sum_{k \in K} |g[i-k]|^2 > 0$ for all $i = 0, 1, \dots, N-1$.
	
	b) (Full set of translations) If $K = \Z_N$, then $\G$ is diagonal-equivalent. Furthermore, $\G$ is a frame for $\C^N$ if and only if $\dsize\sum_{\ell \in L} |\hat{g}[j-\ell]|^2 > 0$ for all $j = 0, 1, \dots, N-1$.
\end{Cor}

Here $\hat{g}$ denotes the discrete Fourier transform of the window vector $g$.   We have recalled the definition and the main properties of the discrete Fourier transform in Section 2.2.

  Corollary \ref{full-sampling}  shows that  the frame operator matrix of $\G$ is diagonal    when  the set of modulations $L$ is  the entire group $\Z_N$. When $L$ is a proper subgroup   of $\Z_N$,   suitable window vectors can still make the frame operator matrix diagonal, yielding optimal computational efficiency:

\begin{Thm}\label{T-Interlace}
 Let $ \mathcal G=\mathcal G(g,L\times K)$  be a Gabor system in $ \mathbb C^N$ , where $ L$  is a modulation subgroup of order $ r$  and $ K$  is a translation subgroup of order $ p>1$ . Set $ M=N/p$ .
	Suppose that $ h_0,\ldots,h_{M-1}\in\mathbb C^p$  satisfy
	$$
	\sum_{t=0}^{p-1}h_u[t]\overline{h_v[t+q]}=0
	\qquad
	\text{for all }u\ne v\text{ and }q\in\mathbb Z_p,
	$$
	where the vector indices are taken modulo $ p$ . Define $ g\in\mathbb C^N$  by
	$$
	g[u+lM]=h_u[l],
	\qquad 0\le u<M,\quad 0\le l<p.
	$$
	If
$
	\gcd\!\left(p,\frac Nr\right)=1,
$
	then the frame operator $ S_{\mathcal G}$  is diagonal.
\end{Thm}

 The construction of the vector $g$ in Theorem \ref{T-Interlace} requires finding $M$ mutually orthogonal vectors in $\mathbb{C}^p$that are mutually  orthogonal for all cyclic shifts. This construction is only possible when $M \le p$, which imposes the necessary dimensional constraint $N \le p^2$.

When  the modulation set $L$ is not a subgroup of $\Z_N$, the  frame operator becomes significantly more complex; However,  if $L$ satisfies certain geometric conditions, we can still guarantee that   the frame operator matrix is sufficiently sparse. We discuss this in section 4.
   
 \medskip
Our paper is organized as follows. Section 2 introduces the definitions and preliminary results required for our proofs. In Section 3, we establish the structural lemmas that characterize the structure of the frame operator matrix  of a  Gabor system  $\G=\G (g,\, L \times K)$ whenever either the modulation set $L$ or the translation set $K$ is a subgroup of $\Z_N$. In Section 4 we  utilize cyclotomic polynomials and tiling conditions to demonstrate how the arithmetic structure of the modulation set forces certain diagonals of the frame operator  matrix to vanish. Finally, in Section 5, we provide the proofs of our main theorems and corollaries.

\section{Notation and Preliminaries}

 Let $N > 1$ be an integer. We denote vectors in $\C^N$ as  $v=(v_0,\, v_1, ..., v_{N-1})$. 
 When   indices are understood modulo $N$, we let   $v = (v[0], v[1], \dots, v[N-1])$.   We  denote with $\l v, \, w\r = \sum_{j=0}^{N-1} v_j \overline{w_j}\ $   the inner product of vectors  $v$,  $w\in\C^N$, and  with  $\|v\| = \sqrt{\l v, v \r}$ the standard norm of $v$.

Given a matrix $A = \{a_{i,j}\}_{i,j=0}^{N-1}$ and an integer $s\in~ [-N+1, N-1] \cap \Z$, we define the $s^{th}$ diagonal of $A$ as:
$$ D_A(s) := \{a_{i,j} : j-i = s\} $$ 
The matrix $A$ has $2N-1$ diagonals; $D_A(0) = \{a_{j,j}\}_{j=1}^{N}$ represents the main diagonal, and the diagonals $D_A(N-1)$ and $D_A(-N+1)$ consist of the elements $a_{0,N-1}$ and $a_{N-1,0}$, respectively.

We denote block-diagonal matrices by $B = \diag(B_1, \dots, B_k)$, where  each  $B_j\in \C^{s_j\times s_j}$ 
is a diagonal block. 

\subsection{ Basic frame theory} We have used the textbooks \cite{Christensen, Heil} for the definitions and the properties of frames  mentioned  in this section.

A set ${\cal F}=\{v_j\}_{j\in J}$ in a Hilbert space $ H $ is a {\it frame} if there exist constants $A,\ B>0$ for which the inequality 
\begin{equation}\label{frame-def-ineq}
	A\|x\|^2 \leq	\sum_{j\in J}    |\l x, v_j\r|^2 \leq B\|x\|^2   
\end{equation}
holds for every $x\in H$.  
The inequality \eqref{frame-def-ineq} is called the {\it frame inequality} and the constants $A, B$ are the {\it frame constants} of ${\cal F}$. When $A=B$, we say that the frame is {\it tight.} 

Let ${\cal F}=\{v_j\}_{j=1}^m\subset\mathbb{C}^N$ be a finite set  of vectors, and let $\mathcal E=\{e_j\}_{j=1}^N$ denote the standard orthonormal basis of $\mathbb{C}^N$.  

The {\it frame operator} of ${\cal F}$ is $ S_{\mathcal F}:\mathbb{C}^N\to\mathbb{C}^N$,  
$$ \dsize  S_{\mathcal F}\, x = \sum_{k=1}^m \langle x, v_k\rangle\, v_k. $$ 

The set $\mathcal F$ is a frame if and only if $S_{\cal F}$    is invertible, and the optimal  frame constants of $\mathcal F$ are the maximum and minimum eigenvalue of $S_{\cal F}$. 

We use the notation  $S_{\mathcal F}$ also to denote  the matrix of the frame operator  (or just  "the   frame matrix " for simplicity).  
Recall that   $S_{\mathcal F} =FF^{*}$, where $F$ is the  matrix whose columns are  the vectors of the frame, also called  the {\it synthesis matrix} of the frame.


\subsection{ 
	The discrete Fourier transform} 

Given a vector \newline
$ 
x=(x_0,\, x_1, ..., x_{N-1}) \in\mathbb{C}^N,
$ 
its normalized discrete Fourier transform  (DFT) is the vector $ \hat x= (\hat x_0,\hat x_1,\dots,\hat x_{N-1})  \in \mathbb{C}^N$  defined by
\begin{equation}\label{eq:dft-def}
	\hat x_k = \frac{1}{\sqrt N}\sum_{j=0}^{N-1} x_j\, e^{-\frac{2\pi i kj}{N} }, 
	\qquad k=0,1,\dots,N-1.
\end{equation}

The transformation \eqref{eq:dft-def} is linear, and can be expressed as a matrix multiplication.  
The   normalized \emph{DFT matrix} is the  $N\times N$ matrix
$$
\F_N := \frac{1}{\sqrt N}\big( \zeta_N^{-kj} \big)_{k,j=0}^{N-1},
$$
where  
$
\zeta_N = e^{  \frac{2\pi i}{N}}.
$
Thus, for every $x\in\C^N$, we   have that $\hat x= \F_Nx$.

A straightforward computation shows that
$
\F_N^{*} \F_N =   I_N,
$
where $\F_N^{*}$ denotes the Hermitian transpose of $\F_N$, and $I_N$ is the $N\times N$ identity matrix.   Thus, the DFT is a unitary transformation that preserves  inner products and norms.

It is well known that the   modulations of a vector $x$  circularly shift  its  DFT in frequency, and vice versa. That is,  
	for every $x\in\C^N$ and every $k_0,\, n_0\in\Z_N$, we have
	$$
	\F_N M_{k_0} x= T_{k_0}\F_Nx;\qquad \F_N T_{n_0} x= M_{n_0}\F_Nx.
	$$
Since $\F_N$ is unitary,  a  Gabor system $\G=\G(g,\,L\times K) $ is a frame  for $\C^N$ if and only if the same is true of 
\begin{equation}\label{e-equiv} \F_N\G=  \G(\hat g,\, K \times L).\end{equation}

The following proposition relates  the frame property of the  Gabor system $\G=\G(g,L\times K)$ to the support  of the window vector $g$ and its Fourier transform. 
  
 \begin{Prop}\label{P-rank}
 	Let $\G=\G(g,\,L\times K)\subset \C^N$. 
 	%
 	If  either \begin{enumerate}
 		\item[(a)] $g$ has fewer than $\frac{N}{|K|}$ nonzero components, or
 		
 		\item[(b)] $\hat g=\F_Ng$ has fewer than $\frac{N}{|L|}$ nonzero components,
 	\end{enumerate}
 	
 	then $\G$ is not a frame for $\C^N$.

 \end{Prop}
 
 \begin{proof}  By \eqref{e-equiv}, it  suffices to prove part (a).
 Let
 $$
  \supp(g) :=\{j\in \mathbb Z_N:\ g[j]\neq 0\}.
 $$ 
Assume that
 $ 
 |\supp(g)|<\frac{N}{|K|}.
 $ 
 For each $k\in K$, the translated vector $T_k g$ is supported on the set
 $ 
 \supp(g)+k.
 $ 
Since the modulations  do  not affect the support, we can conclude that   the support of every vector of  $\G$ is in the set 
 $$
 Z=\bigcup_{k\in K}(\supp(g)+k).
 $$
Since
 $ 
 |Z|\leq |\supp(g)||K|<N, 
 $ 
 we have that  $Z$ is a proper subset of $\mathbb Z_N$, and every vector in $\G$
 vanishes outside $Z$. Consequently, the vectors in $\G$ cannot span $\C^N$,
 and so $\G$ is not a frame for $\C^N$.
 \end{proof}

More details on the  DFT can be found e.g in \cite[Chapt 7]{SteinShakarchi}.

\subsection{ Circulant and block-circulant matrices} 
 
We refer to \cite{Davis} for  the  definitions and the results on circulant matrices mentioned in this section.

Let $\vec c=(c[0],c[1],\dots,c[n-1])\in\C^n$. The $n\times n$ \emph{circulant matrix} identified by the vector $\vec c$  is
$$
C= C_{\vec c} :=
\begin{pmatrix}
	c[0] & c[n-1] & \cdots & c[1] \\
	c[1] &c[0]  & \cdots & c[2] \\
	\vdots & \vdots   & \ddots & \vdots \\
	c[n-1]  & c[n-2]  & \cdots & c[0]
\end{pmatrix}.
$$
The columns of $C$ are cyclic permutations of the vector $\vec c$. Circulant matrices are    special  Toeplitz matrices.

The  circulant matrices form a commutative algebra,  and    the Hadamard (componentwise)  product of circulant  matrices  of the same size is circulant.  Specifically, 
if 
$C$ and $D$ are $n\times n$ circulant matrices with associated vectors $\vec c$ and $\vec d$,  the matrix 
$C\odot D$ is  the circulant matrix identified by  the vector $\vec c\odot \vec d$.

A fundamental property of the   $n\times n$ circulant matrices is that they are diagonalized by the discrete Fourier transform matrix $\F_n  $.   
 The eigenvalues of 
 $ C_{\vec c}$  are  given by 
$$ 
\lambda_j(C)=P_{\vec c}(\zeta_n^{j}), \qquad j=0,\dots,n-1.
$$
where  
$ 
P_{\vec c}(z)=c[0]+c[1]z+\cdots+c[n-1]z^{n-1}.
$ 

An eigenvector associated to $\lambda_j$ is 
$ 
\vec v_j=\frac{1}{\sqrt n}(1,\zeta_n^{-j},\zeta_n^{-2j},\dots,\zeta_n^{-(n-1)j}).
$ 

\medskip

A natural generalization of circulant  matrices is obtained by replacing the scalar entries with matrices  and preserving the same cyclic structure at the   block level.

 Let $ m,k \in \mathbb{N} $ and let
$
B_0, B_1, \dots, B_{m-1} \in \mathbb{C}^{k \times k}.
$
The \emph{block-circulant matrix} determined by the ordered blocks $ (B_0,\dots,B_{m-1}) $ is the  $ mk \times mk $ matrix
\begin{equation}\label{e-block-circ}
	C = C(B_0,\dots,B_{m-1}) :=
	\begin{pmatrix}
		B_0 & B_{m-1} & \cdots & B_1 \\
		B_1 & B_0     & \cdots & B_2 \\
		\vdots & \vdots & \ddots & \vdots \\
		B_{m-1} & B_{m-2} & \cdots & B_0
	\end{pmatrix}.
\end{equation}

A block-circulant matrix is not, in general, circulant. However, any $ n \times n $ circulant matrix with $n = mk $ can be viewed as a block-circulant matrix with $ k \times k $ blocks.

The Hadamard product of two  block-circulant matrices  with the same block structure is  block circulant,  and so the Hadamard product of a $n\times n$ circulant matrix, with $n=mk$,   and a block-circulant matrix with blocks of size $k\times k $ is block-circulant  with blocks of size $k\times k$.

\medskip

We recall the definition of Kroneker product of   matrices.

Let $A=(a_{ij}) \in \mathbb{C}^{m\times n}$ and
$B \in \mathbb{C}^{k\times \ell}$. The \emph{Kronecker product}
of $A$ and $B$ is the matrix
$ 
A \otimes B \in \mathbb{C}^{mk \times n\ell}
$ 
defined by
$$
A \otimes B =
\begin{pmatrix}
	a_{11}B & a_{12}B & \cdots & a_{1n}B \\
	a_{21}B & a_{22}B & \cdots & a_{2n}B \\
	\vdots & \vdots & \ddots & \vdots \\
	a_{m1}B & a_{m2}B & \cdots & a_{mn}B
\end{pmatrix}.
$$

The Kronecker product   satisfies the mixed-product property
$$
(A\otimes B)(C\otimes D) = (AC)\otimes (BD),
$$
whenever the matrix products $AC$ and $BD$ are defined.

A block-circulant matrix is unitarily equivalent to a block-diagonal matrix via a block analogue of the discrete Fourier transform defined as follows.
Let $ m,k \in \mathbb{N}  $, and let
$ 
\F _m := \frac{1}{\sqrt{m}}\big(\zeta_m^{\,j\ell}\big)_{j,\ell=0}^{m-1}
$ 
be the  DFT  matrix on $ \mathbb{C}^m $. The \emph{block Fourier transform} is the matrix
$$
\F_{m \otimes k} := \F_m \otimes I_k \in \mathbb{C}^{mk \times mk},
$$
where $ I_k $ denotes the identity matrix on $ \mathbb{C}^k $.
  The blocks of $\F_{m \otimes k} $ are scalar multiples of the identity matrix $I_k$.

The following important proposition follows from Theorem~5.6.4 in \cite{Davis}.

\begin{Prop}\label{T-Block-FT}
	The matrix $\F_{m \otimes k} $ is unitary, and it block-diagonalizes all $mk\times mk$ block-circulant matrices with $k\times k $ blocks. Specifically,
$$
	C(B_0, \dots, B_{m-1}) 
	= (\F_m^* \otimes I_k)\, \diag(D_0,\dots,D_{m-1})\, (\F_m \otimes I_k),
$$
	where
	$ \dsize
	D_\ell = \sum_{j=0}^{m-1} \zeta_m^{\,j\ell} B_j,\ $ for $\ell = 0,\dots,m-1.
$ 
\end{Prop}


\section{Structural Lemmas}
 
In this section, we establish the foundational algebraic properties of the frame operator matrix $S_\G$ for a finite Gabor system $\G=\G(g, L \times K)$. These lemmas will serve as the primary tools for proving our main results in the subsequent sections.

Our first lemma shows that, when the time-frequency set is a Cartesian product, the frame operator matrix can be expressed as the Hadamard product of two matrices: one depending only on the translations and the other depending only on the modulations of the window vector $g$.

\begin{Lemma}\label{L-Hadam}
	Let $\G= \G(g, L \times K)\subset   \C^N $ be a Gabor system. 
	  Let \newline $\M _L $  and $\T_K  $  be the $N\times N$ matrices whose elements are 
\begin{equation}\label{e-matrix-T-M}
\M_L[ i,j]= \sum_{\ell \in L } \zeta_N^{\ell(i-j)}, \quad \T_K[ i,j]=	\sum_{k \in K} g[i-k] \,\overline{g [j-k]},
\end{equation}
with $ i, j=0,\, ...,\, N-1$.	The frame matrix of $\G$ is  
	$$S_\G = \M_L \odot \T_K,$$ where $\odot$ denotes the Hadamard (componentwise) product.
\end{Lemma}

Recall that   $\zeta_N=e^{2\pi i/N}$. We will term the matrices $\M_L$ and $\T_K$   {\it modulation matrix} and {\it translation matrix} of the system    $\G$.

\begin{proof}
	Recall that the frame matrix of $\G$ is the $N \times N$ matrix $ FF^{*}$, where $F$ is the matrix whose columns are the vectors of the frame.   The $(i,j)$ entry of the frame matrix $FF^*$ is given by: 
	\begin{align*}
		S_\G[i,j] &= \sum_{(k,\ell)\in \Lambda}\left(M_\ell T_k g\right)[i] \overline{\left(M_\ell T_k g\right)[j]} \\
		&= \sum_{(k,\ell)\in \Lambda}  \zeta_N^{i\ell}  g[i-k]\overline{ \zeta_N^{j\ell} g[j-k]} \\
		&= \sum_{k\in K} \sum_{\ell \in L } \zeta_N^{(i-j)\ell} g[i-k] \overline{g[j-k]} \\
		&= \underbrace{\left( \sum_{\ell \in L } \zeta_N^{(i-j)\ell} \right)}_\text{$\M_L{[i,j]}$} \underbrace{\left( \sum_{k \in K} g[i-k] \,\overline{g [j-k]}\right)}_\text{$\T_K[i,j] $}
	\end{align*}
	as required.
\end{proof}

\begin{Rem}
   It is well known that   the rank of the  Hadamard product of two matrices  does not exceed the product of the ranks of these   matrices (see e.g.  \cite[Chapt. 5]{HornJohnson}).
Thus,  for the   set $ \G=\G(g, L\times K)\subset \C^N$, we have that
$$\rank\, S_{\G}\leq  \rank \M_L \, \rank \T_K   \leq |K| \, |L|.$$

 It is easy to verify that  $\rank \M_L = |L|$, but we can  find window vectors  $g$ for which $\rank \T_K <|K|$. For instance, when  $g=(1,1,\, ...,\, 1)$, we have $\rank \T_K=1$ for  every translation set $K$.
 \end{Rem}

Lemma \ref{L-Hadam} shows that the structure of the frame operator $S$ is 
dictated by the structures of $M_L$ and $T_K$. We prove that these matrices are circulant under specific conditions on  the sets $L$ and $K$.

\begin{Lemma}\label{L-circulant-prod}
	a) Let $\G=\G(g,\, L\times K)$. The modulation matrix $\M_L$ is circulant, and the associated vector is 
	$$\vec m= \left(|L|, \, \sum_{\ell \in L } e^{2 \pi i\ell/N},\, \dots, \, \sum_{\ell \in L } e^{2 \pi i\ell(N-1)/N}\right).$$
	
	b) If $K=\Z_N$, also the translation matrix $\T_K$ is circulant, and the associated vector is 
	$$\vec t= \left( \sum_{k \in \Z_N} | g [-k]|^2, \,  \sum_{k \in \Z_N} g[1-k] \,\overline{g [-k]},\, \dots, \sum_{k \in \Z_N} g[N-1-k] \,\overline{g [-k]}\right).$$
\end{Lemma}

\begin{proof} 
	a) The entries $\M_L[i, j]$ depend only on $i-j$, and for every $s\in\{0, 1, \dots, N-1\}$, 
	$$ \M_L[N-s, 0] = \sum_{\ell\in L} \zeta_N^{\ell(N-s)} = \sum_{\ell\in L} \zeta_N^{\ell( -s)} = \M_L( -s). $$
	This establishes that $\M_L$ is a circulant matrix.
	
	b) If $K=\Z_N$, we have that 
	\begin{align*}
		\T_K(i+s, j+s) &= \sum_{k \in \Z_N} g[i-(k-s)] \,\overline{g [j-(k-s)]} \\
		&= \sum_{k' \in \Z_N} g[i-k'] \,\overline{g [j-k']} = \T_K(i , j ).
	\end{align*}
	Thus, the matrix $\T_K$ is Toeplitz. Since 
	$$ \T_K(N-s, 0) = \sum_{k \in \Z_N} g[N-s-k] \,\overline{g [-k]} = \sum_{k \in \Z_N} g[ -s-k] \,\overline{g [-k]} = \T_K(-s), $$
	we have proved that it is circulant.
\end{proof}	

When the translation set $K$ is a proper subgroup of $\Z_N$, the translation matrix is no longer strictly circulant, but retains a block-circulant structure.

\begin{Lemma}\label{L-trans-block-circ}
	Let $K$ be a subgroup of translations of order $p$, and let $M = \frac{N}{p}$. The translation matrix $\T_K$ is block-circulant with blocks of size $M \times M$. That is, for all $i, j \in \mathbb{Z}_N$:
\begin{equation}\label{e-per-3} \T_K[i+M, j+M] = \T_K[i, j]. \end{equation}
	
\end{Lemma}
Note that  the relation \eqref{e-per-3}   implies that the $N \times N$ matrix is determined entirely by its   the top row of  $M\times M$ blocks.
\begin{proof}
	Recall that the unique  subgroup of $\mathbb{Z}_N$ of order $p$ is generated by the element $M = \frac Np$. Thus, we can write 
	$$K =\langle M\rangle = \{0, M, 2M, \dots, (p-1)M\}.$$ The $(i,j)$-th entry of the translation matrix is given by:
	$$ \T_K[i, j] = \sum_{k \in K} g[i-k] \overline{g[j-k]}. $$
	
	Now consider the entry at indices shifted by $M$:
	$$ \T_K[i+M, j+M]= \sum_{k \in K} g[i+M-k] \overline{g[j+M-k]}. $$
	
	Let $k' = k - M$. Since $M \in K$ and $K$ is a subgroup, the map $k \mapsto k-M$ is a bijection from $K$ onto itself. Substituting $k = k' + M$ into the sum, we obtain:
	\begin{align*}
		\T_K[i+M, j+M] &= \sum_{k' \in K} g[i+M-(k'+M)] \overline{g[j+M-(k'+M)]} \\
		&= \sum_{k' \in K} g[i-k'] \overline{g[j-k']}  
		 = \T_K[i, j]
	\end{align*}
	which completes the proof.
\end{proof}

Our next lemma shows how  specific   window vectors $g$  can produce highly sparse matrices. 

\begin{Lemma}\label{L-trans-sparsity} 
	Let $K$ be a subgroup of translations of order $p$, and let $M = \frac{N}{p}$.
	
	a)  The translation matrix $\T_K$ is block-circulant with $p$ blocks of size $M \times M$.  

b) Suppose that $ g$  is obtained by interlacing vectors $ h_0,\ldots,h_{M-1}\in\mathbb C^p$ , in such a way that  $ 
	g[u+lM]=h_u[l]$, with $
  0\leq u<M,\quad 0\leq l<p.
$
Further, assume that 
$$
	\sum_{t=0}^{p-1}h_u[t]\overline{h_v[t+q]}=0
	\qquad
	\text{for all }u\ne v\text{ and }q\in\mathbb Z_p,
	$$
	where the vector indices are taken modulo $ p$ .
	
	Then $\T_K[i,j]=0$ whenever
	$i\not\equiv j\pmod M$.
	\end{Lemma}
 
 \begin{proof}
	
	Lemma \ref{L-trans-block-circ} shows that the translation matrix $\T_K$ is completely determined by  the        $M \times M$   blocks of the top row  which we will call $B_0...B_{p-1}$.    That is,  $B_q[u,v] = \T_K[u, v+qM]$ where $0 \leq u,v \leq M-1$ and $q\in \{0... p-1$\}.
Writing $k=-lM$, $l=0,\dots,p-1$ (which runs over $K$ since $K$ is a group), we get
$$B_q[u,v]=\T_K[u,v+qM] 
	 =\sum_{l=0}^{p-1}g[u+lM]\,
		\overline{g[v+(l+q)M]} 
		 =\sum_{l=0}^{p-1}h_u[l]\,
		\overline{h_v[l+q]}.
$$
	Here the vector indices are taken modulo $ p$ . By assumption, this sum vanishes whenever $ u\ne v$  and so every block $ B_q$  is diagonal. The block-circulant structure implies that all blocks are diagonal, so
	$  
\T_K[i,j]=0$
 {whenever } $i\not\equiv j\pmod M.
	$ 
\end{proof}

By Lemma \ref{L-Hadam},  the frame operator matrix of $\G=\G(g,\, L\times K)$ is $S_\G = \M_L \odot \T_K$. Consequently, any zero entry in either $\M_L$ or $\T_K$ forces a zero in $S_\G$.

In the next lemma we evaluate the exponential sums  in the definition of $\M_L$. 

\begin{Lemma}\label{L-subgroup-sum}
	Let $L$ be a subgroup of $\mathbb{Z}_N$ of order $r$. For any integer $d$:
	$$ \sum_{\ell \in L} \zeta_N^{\ell d} = 
	\begin{cases} 
		r & \text{if } r \text{ divides } d \\ 
		0 & \text{if } r \text{ does not divide } d. 
	\end{cases} $$
\end{Lemma}

\begin{proof}
	Since $L$ is a subgroup of order $r$, it is generated by the element $N/r$. Every $\ell \in L$ can be written as $\ell = k N/r$, with $k~=~0, 1, \dots, r-1$. Substituting this into the sum yields:
	$$ \sum_{\ell \in L} \zeta_N^{\ell d} = \sum_{k=0}^{r-1} e^{2\pi i (k N/r) d / N} = \sum_{k=0}^{r-1} e^{2\pi i k d / r}. $$
	If $r$ divides $d$, then $e^{2\pi i k d / r} = 1$ for all $k$, and the sum evaluates to $r$. If $r$ does not divide $d$, we use the formula for a finite geometric series:
	$$ \sum_{k=0}^{r-1} (e^{2\pi i d / r})^k = \frac{1 - e^{2\pi i d}}{1 - e^{2\pi i d / r}} = 0 $$
	which completes the proof.
\end{proof}

Finally, we state a   matrix-theoretic lemma that allows to prove most of our  results.  Version of this lemma  appear  in the literature,  but since we could not find a reference to the complete result, we have proved it in the Appendix.

\begin{Lemma}\label{L-main-block}
	Let $A=\{a_{i,j}\}_{i, j=1,\dots,N}$, and let $U \subset [-N+1, N-1]$ be the set of the indices of diagonals of $A$ that are not identically zero. Assume that $0 \in U$ and if $d \in U$, then $-d \in U$. Assume also that:
	$$ \ell := \gcd(N, \{d \in U\}) > 1. $$
	Then, $A$ is orthogonally equivalent to a block-diagonal matrix with $\ell$ blocks, all of size $N/\ell$. The element $(i,j)$ of the block $B_s$ is:
	$$ b_{i,j}^{(s)} := a_{i\ell+s, j\ell+s} $$
	where $i, j = 0, 1, \dots, \frac{N}{\ell}-1$ and $s = 0, 1, \dots, \ell-1$.
\end{Lemma}

 \section{Cyclotomic structures}
 Let $\G=\G(g,  L \times K)\subset \C^N$ be a   Gabor system. When neither $L$ nor $K$ is a subgroup of $\Z_N$, the associated frame operator matrix cannot be reduced to block-diagonal form in an obvious way.
 
 Nevertheless, one may ask under which conditions certain diagonals of the frame matrix are identically zero. 
 
 Cyclotomic polynomials arise naturally in this context: 
 %
 %
 Let $\zeta_n := e^{2\pi i /n}$ be a primitive $n$-th root of unity. The cyclotomic polynomial of index $n\in\N$ is defined as
 $$
 \Phi_n(z):=\prod_{\substack{1\leq j\leq n \\ \gcd(j,n)=1}} (z-\zeta_n^j).
 $$
 The degree of $\Phi_n(z)$ is Euler's totient function $\varphi(n)$, which counts the integers $1\le j<n$ that are coprime with $n$.
 
 The polynomials $\Phi_n(z)$ have integer coefficients and are irreducible in $\Q[z]$ (see, e.g., \cite{Sanna}, \cite[Chapter~9]{Barbeau}). Moreover,
 \begin{equation}\label{e-prod-cyclo}
 	z^n-1=\prod_{d\mid n}\Phi_d(z).
 \end{equation}
 Understanding when a  certain  polynomial is divisible by a cyclotomic polynomial $\Phi_n(z)$ plays an important role in many areas of pure and applied mathematics. A necessary and sufficient condition for such divisibility is given in \cite[Theorem~1.1]{DC-Laporta}; see also \cite[Prop.~5.5]{LabaLond}.
 
 \medskip
 
 For a subset $A\subset \Z_N$, define its \emph{characteristic polynomial} by
 $$
 P_A(z)=\sum_{a\in A} z^a.
 $$
 Let $\G = \G(g, L \times K) \subset \mathbb{C}^N$ be a     Gabor system.  
  The following theorem shows that the divisibility properties of $P_L(z)$ by cyclotomic polynomials impose strong   constraints on the frame matrix of $\G$,  forcing certain diagonals to vanish identically.
 
 \begin{Thm}\label{T-cyclo} 
 	Suppose there exists $d>1$  that divides $N$  such that the cyclotomic polynomial $\Phi_d(z)$ divides $P_L(z)$. Then the diagonals of  the frame matrix of $\G$ with indices $\pm s\frac{N}{d}$, where $\gcd(s,d)=1$, are identically zero.
 \end{Thm}

\begin{proof}  We show that the diagonals of the matrix    $\M_L$ with indices $\pm s\frac{N}{d}$, where $\gcd(s,d)=1$, are  identically  zero. By Lemma \ref{L-Hadam},the same is true of the diagonals of the  matrix  $S_\G$.
	
	 Recall that by Lemma \ref{L-circulant-prod}, the  modulation matrix $\M_L$ is circulant and  it is identified by the vector
	$\vec m= (m_0,\, m_1,\, ... m_{N-1})$, with $m_j~=~\sum_{\ell\in L} \zeta_N^{\ell j}$.
	We show that  if $s\in \{0,\, 1,\,... N-1\}$ satisfies gcd$(s,d)=1$,  then $ 
	m_{s\frac{N}{d}} = 0.
	$ 
	Indeed, 
	$$
	m_j=\sum_{\ell\in L} \zeta_N^{\ell j}
	= P_L\big(\zeta_N^{  j}\big).
	$$
	
	Since   $d $ divides $ N$, we can write $N=dd'$, { with $d'\in\N$}.  For every $s$ with $(s,d)=1$,  we have that $e^{2\pi i s/d}$ is a primitive $d$-th root of unity, and hence a root of $\Phi_d(z)$ and also of $P_L(z)$. Thus, 
	$$
	m_{s\frac Nd}=m_{sd'} =   P_L(
	e^{2\pi i (sd')/N})=
	P_L(e^{2\pi i s/d}) = 0.
	$$
\end{proof}
\vskip - .3 cm
 Divisibility by cyclotomic polynomials also arise naturally in problems related to tiling   of finite cyclic groups.

 Let $A\subset \Z_N$. We say that $A$ \emph{tiles} $\Z_N$ by translation if there exists a set   $B=\{b[0],\dots,b[m-1]\}\subset \Z_N$ such that
 $ 
 A+b[j] \cap A+b[k]=\emptyset \quad \text{if } j\neq k,
 $ 
 and
 $ \dsize 
 \bigcup_{j=0}^{m-1}(A+b[j])=\Z_N.
 $ 
 In this case we write $\Z_N = A \oplus B$.
 
 All subgroups of $\Z_N$ are of the form
 $ 
 d\,\Z_N=\{kd : k=0,\dots,\frac Nd-1\}
 $,  with $ d\mid N$
 and hence
 \begin{equation}\label{e-tile}
 	\Z_N = d\,\Z_N \oplus \{0,1,\dots,d-1\}.
 \end{equation}
 We may also denote the subgroup $d\,\Z_N$ by $\langle d\rangle$ when there is no risk of confusion.
 
 The sets that tile $\Z_N$ are not limited to subgroups. Indeed, $A$ tiles $\Z_N$ with translation set $B$ if and only if $B$ tiles $\Z_n$ with translation set $A$.   For instance, \eqref{e-tile} shows that also  the set $\{0,1,\dots,d-1\}$  tiles $\Z_n$, but this is not a subgroup  of $\Z_N$ if $N>d$. There are also more intricate sets that tile $\mathbb{Z}_N$ under certain conditions on $N$. For further references, see~\cite{Szabo1985, CovenMeyerowitz1999}.

 It is proved  in \cite[Theorem~B1]{CovenMeyerowitz1999} and also in \cite{LabaLond}, that if $A$ tiles $\Z_N$ with translation set $B$, then
  $P_A(z)$ is divisible by cyclotomic polynomials $\Phi_d(z)$ for some  $d$ that divides $N$.  
 
This observation combined  with Theorem ~\ref{T-cyclo}    yields the following corollary.

 \begin{Cor}\label{C-tiling}  Assume that  $L\subset \Z_N$, $|L|>1$  tiles $\Z_N$  by translation. Then, there exist an integer $d>1$ that divides $N$  such that
 	the diagonals of  the frame matrix of $\G$ with indices $\pm s\frac{N}{d}$   with gcd$(s,d)=1$   are identically zero.\end{Cor}

 The next theorem generalizes the result of Theorem~\ref{T-Interlace} when the set of modulations is not a subgroup.

 \begin{Cor}\label{C-cyclo-Interlace}
 	Let $\G=\G(g,L\times K)$, where $K$ is a subgroup of $\Z_N$ of order $p>1$,
 	$M=N/p$, and $g$ is constructed as in Theorem~\ref{T-Interlace}. Let
 	$D=\{d>1:\ d\mid N,\ \Phi_d(z)\mid P_L(z)\}$. If for every $j=1,\dots,p-1$
 	there exists $d\in D$ such that $s_j:=\frac{jMd}{N}=\frac{jd}{p}$ is an
 	integer with $\gcd(s_j,d)=1$, then $S_\G$ is diagonal.
 \end{Cor}
 
 \begin{proof}
 	By Lemma~\ref{L-trans-sparsity}, $\T_K[i,j]=0$ unless $M\mid (j-i)$, so the
 	only diagonals of $S_\G=\M_L\odot\T_K$ that can be nonzero have indices
 	$0,\pm M,\dots,\pm(p-1)M$. For $1\le j\le p-1$ we have $jM=s_j\frac{N}{d}$
 	with $\gcd(s_j,d)=1$, so by Theorem~\ref{T-cyclo} the diagonals $\pm jM$ of
 	$\M_L$, and hence of $S_\G$, vanish.
 \end{proof}
 The following example shows that   a set $L$  that is not a subgroup can satisfy the  condition described in Corollary~\ref{C-cyclo-Interlace}.  

\noindent
{\it Example.}
Let $N=12$ and $K=\langle 3\rangle=\{0,3,6,9\}$, so that $p=4$ and $M=3$.
Let $L=\{0,1,2,3\}$, which is not a subgroup of $\Z_{12}$.

Since $P_L(z)=1+z+z^2+z^3=\Phi_2(z)\Phi_4(z)$ and both $2$ and $4$ divide $12$,
we have $2,4\in D$. For $j=1,2,3$ choose $d=4,2,4$, respectively; then
$s_j=\frac{jMd}{N}=\frac{jd}{4}$ equals $1,1,3$, each coprime to the corresponding $d$.

For the window vector, take $h_u=(1,i^{u},i^{2u},i^{3u})\in\C^4$ $(u=0,1,2)$, and let
$g$ be obtained by interlacing them in such a way that $g[u+3l]=h_u[l]$. For $u\ne v$ and $q\in\Z_4$,
$$
\sum_{t=0}^{3}h_u[t]\,\overline{h_v[t+q]}=i^{-vq}\sum_{t=0}^{3}i^{(u-v)t}=0,
$$
since $4$ does not divide $u-v$. Thus $g$ is constructed as in Theorem~\ref{T-Interlace}, and by
Corollary~\ref{C-cyclo-Interlace} the matrix $S_\G$ is diagonal. In fact
$S_\G=|L|\,\|h_u\|^2 I=16\,I$, so $\G$ is a tight frame.

 \section{Proof of the main theorems}
 
 In this section  we prove the  theorems and corollaries introduced in Section 1. Our proofs rely on the structural lemmas established in Section 3.
 
 \subsection{Proof of Theorem \ref{T-Main}}
 
 
 \noindent {\it Case 1: The modulation set $L$ is a subgroup.} \\
 Assume $L$ is a subgroup of $\Z_N$ of order $r$. Since $L$ is generated by $N/r$, every $\ell \in L$ can be written as $\ell= k N/r$, with $k=0, \dots, r-1$.
 We show that the elements of the $\frac{2N}{r} - 1$ diagonals of $\M_L$ labeled with indices $0, \pm r, \dots, \pm (\frac{N}{r}-1) r$ are equal to $r$ and the elements of the other diagonals are zero.
 Indeed, letting $d = j-i$, all elements of the diagonal $D(d)=\{ \M_L [i,j]: j-i=d\}$ equal
 $$ m(d):= \sum_{\ell \in L }\zeta_N^{ d\ell }=\sum_{k=0}^{r-1} e^{-2 \pi i d k/r}. $$
 This is a standard sum of roots of unity, which evaluates to:
 $$ m(d) = \begin{cases}
 	0 & \mbox{if $r$ does not divide $d$} \\
 	r & \mbox{if $r$ divides $d$.}
 \end{cases} $$
 Thus, the set of the non-negative indices of the nonzero diagonals modulo $N$ is $H=\{0, r, \dots, (\frac{N}{r}-1)r\} = \langle r \rangle$, and Lemma \ref{L-main-block} allows us to conclude that $\M_L$ is orthogonally equivalent to a block-diagonal matrix via an explicit permutation matrix $P$.
 Because $S_\G = \M_L \odot \T_K$, the non-zero diagonals of $S_\G$ are restricted to be a subset of the non-zero diagonals of $\M_L$. Therefore, the exact same permutation matrix $P$ block-diagonalizes $S_\G$, meaning $\G$ is block-equivalent.
 
 \medskip
 
 \noindent {\it Case 2: The translation set $K$ is a subgroup.} \\
 Assume $K$ is a subgroup of translations of order $p$. By Lemma \ref{L-trans-block-circ}, the translation matrix $\T_K$ is block-circulant with blocks of size $M \times M$, where $M = N/p$. 
 Because $\M_L[i,j]$ depends solely on the index difference $i-j$, the matrix $\M_L$ is   circulant for any set $L$. Since the Hadamard product of a circulant matrix and a block-circulant matrix results in a block-circulant matrix, this structural property is preserved in the frame operator    $S_{\G} = \M_L \odot \T_K$. 
 
 As stated in Section 2.3, a block-circulant matrix is unitarily similar to a block-diagonal matrix via a block-discrete Fourier transform matrix. Therefore, $\G$ is block-equivalent.
 
 \medskip

 The following corollary gives explicit formulas for the blocks arising in the block-diagonalization of the frame operator matrix of $\G$.

\begin{Cor} \label{C-explicit-blocks}
	Let $\G=\G(g,\,L\times K)$ be a finite Gabor system. 
	
	\begin{itemize}
		\item [a)] If  $L$ is a subgroup of $\Z_N$ of order $r$, 
		the frame operator matrix $ S_\G$ is orthogonally equivalent to a block-diagonal matrix consisting of blocks $B_0, \dots, B_{r-1}$.  
		The elements of the blocks $B_s$ are:
		\begin{equation}\label{e-bij} 
			b_{i,j}^{(s)} := S_\G({i r +s,\, j r+s})= r \T_K(i r +s,\, j r+s) 
		\end{equation} 
		where $i,\, j = 0,\, \dots, \frac{N}{r}-1$ and $s=0,\, 1,\, \dots, r-1$.
		
		\item [b)] If  $K$ is a subgroup of $\Z_N$ of order $p$,  
		the frame operator matrix $S_\G$ is block-circulant. Letting $A_0, \dots, A_{p-1}$ denote the $p$ blocks of size $\frac{N}{p} \times \frac{N}{p}$ that form its first block-row, $S_\G$ is unitarily similar to a block-diagonal matrix $\diag(B_0,\, \dots , B_{p-1})$. The blocks $B_s$ are given by: 
		\begin{equation}\label{e-bhat} 
			{B}_s = \sum_{m=0}^{p-1} A_m e^{-2\pi i s m / p}, \quad s = 0, 1, \dots, p-1.
		\end{equation}
		
	\end{itemize}
\end{Cor}

\begin{proof}
	a)  The proof of Theorem \ref{T-Main} establishes that the modulation matrix $\M_L$ is zero everywhere except on diagonals whose indices are multiples of $r$, where the entries evaluate to exactly $r$. 
	Because the frame operator  matrix is defined by the Hadamard product $S = \M_L \odot \T_K$, this element-wise multiplication forces $S_\G$ to inherit the     same sparsity pattern as $\M_L$. 
	By Lemma \ref{L-main-block}, any matrix with this specific sparsity pattern is block-diagonalized by a permutation matrix, with the elements of the new blocks mapped from the original coordinates $(ir+s, jr+s)$. Since $\M_L$ equals $r$ at these coordinates, the Hadamard product directly yields the formula in \eqref{e-bij}.
	
	b)  The proof of Theorem \ref{T-Main} relies on Lemma \ref{L-trans-block-circ} to establish that the translation matrix $T_K$ is   block-circulant. 
	Because $\M_L$ is a circulant matrix, the Hadamard product $  \M_L \odot \T_K$ preserves this block-circulant structure. 
	As established in Section 2.3, any block-circulant matrix is unitarily similar to a block-diagonal matrix. Applying this standard transformation directly produces  the summation formula shown in \eqref{e-bhat} for each resulting diagonal block.
\end{proof}

\subsection{Proof of Corollary \ref{full-sampling}}

a) When $L=\Z_N$, the elements of the modulation matrix $\M_L$ are given by the standard sum of roots of unity:
$$ \M_L[h, j]= \sum_{\ell=0}^{N-1} e^{2 \pi i\ell(h-j)/N} = \begin{cases}
	0 &  \mbox{if $h \neq j$} 
	\cr  
	N & \mbox{if $h = j$.}
\end{cases}
$$
Thus, $\M_L$ is diagonal and by Lemma \ref{L-Hadam}, the frame operator matrix of $\G$ equals $N\,\diag[ \T_K[0,\,0], \T_K[1,\,1],\, \dots, \T_K[N-1,\, N-1]],$ 
where  $\T_K$ is the translation matrix.
Since $$\T_K[j,j]= \sum_{k \in K} g[j-k] \,\overline{g [j-k]} =\sum_{k \in K} |g[j-k]|^2, $$ the proof of part a) is concluded.

\medskip

b) Assume now that $K=\Z_N$. As shown in \eqref{e-equiv}, applying the discrete Fourier transform matrix to $\G$ yields the unitarily equivalent system $\F_N \G(g, L \times \Z_N) = \G(\hat{g}, \Z_N \times L)$. Because this transformed system possesses the full set of modulations $\Z_N $,    we can apply the result of part a) directly. The frame operator of this transformed system is a diagonal matrix whose $j$-th diagonal entry is exactly $N \sum_{\ell \in L} |\hat{g}[j-\ell]|^2$. 

Because $\F_N$ is unitary, the frame operators of the two systems are unitarily similar,    and so  they share the exact same eigenvalues. Therefore, $\G$ is diagonal-equivalent, and its eigenvalues are given explicitly by $\lambda_j = N \sum_{\ell \in L} |\hat{g}[j-\ell]|^2$. The system $\G$ is a frame if and only if all its eigenvalues are strictly positive, yielding the condition $\sum_{\ell \in L} |\hat{g}[j-\ell]|^2 > 0$ for all $j = 0, 1, \dots, N-1$.

\medskip

\noindent \textit{Remark:} In both cases of Corollary \ref{full-sampling} one of the sets constitutes the entire group $\Z_N$ and so the  blocks are  $1 \times 1$. Thus, the frame operator matrix is  diagonal or diagonal-equivalent.

\subsection{Proof of Theorem \ref{T-Interlace}}

 The frame operator $S_\G$ is the Hadamard product of the modulation matrix $\M_L$ and the translation matrix $\T_K$.
	By Lemma \ref{L-subgroup-sum}, the modulation matrix $\M_L$ is possibly nonzero only on diagonals that are multiples of $r$ (the order of the modulation subgroup $L$).

Lemma \ref{L-trans-block-circ} and Lemma \ref{L-trans-sparsity} show that the construction of the  window vector $g$ ensures that the translation matrix $\T_K$ is possibly nonzero only on diagonals whose index is a multiple of $M =\frac Np$, where  $p=|K|$. 
	By Lemma $ \ref{L-trans-sparsity}$ ,
$ 
\T_K[i,j]=0$
{unless } $M\mid(i-j).
$ 
Consequently,
$
\S_{\mathcal G}[i,j]=0$ 
unless   $\operatorname{lcm}(r,M)\mid(i-j).
$
Because $r$ and $ p$  divide $ N$ ,
$$
\operatorname{lcm}\!\left(r,\frac Np\right)
=\frac{N}{\gcd(N/r,p)}=N
$$
For $ 0\le i,j<N$ , the condition $ N\mid(i-j)$  implies $ i=j$ , and so  $ S_{\mathcal G}$  is diagonal.  

Therefore, the frame matrix $S_\G$ is possibly nonzero only on diagonals that are multiples of both $r$ and $\frac Np$, i.e., 
$
S_{\mathcal G}[i,j]=0$ 
unless   $\operatorname{lcm}(r,M)\mid(i-j).
$
Because both $r$ and $ p$  divide $ N$ ,
$$
\operatorname{lcm}\!\left(r,\frac Np\right)
=\frac{N}{\gcd(N/r,p)}=N
$$
For $ 0\le i,j<N$, the condition $ N\mid(i-j)$  implies $ i=j$ , and so  $ S_{\mathcal G}$  is diagonal.

 \section*{Appendix. Proof of Lemma 3.6}

	We denote with $H$ the subgroup of $\mathbb{Z}_N$ generated by $U$. Assume first that $U=H$.
	If $|H|=r$, from basic linear algebra results it follows that $H$ is the cyclic group generated by $\frac{N}{r}$.
	We use the notation  
	$$\mbox{$ H := \left\langle \frac{N}{r} \right\rangle  =\left\{ k\frac{N}{r}, \ k=0, 1, \dots, r-1\right\}$} .$$  
	
	We prove that there exist an orthogonal matrix $P$ such that $P^* A P$ is block diagonal.
	Define a ``direct link'' relation $\sim_H$ on the set $\{0,\dots,N-1\}$ as follows: 
	$$i \sim_H j \iff i-j \in H$$
	That is, two indices $i$ and $j$ are directly linked if there exists a nonzero element in the diagonal $D_{i-j}$. The relation $\sim_H$ is reflexive (because $0 \in U$) and symmetric (because $d \in U \implies -d \in U$) and it is also transitive because $i \sim_H j$ and $j \sim_H k$ imply $i-j$ and $j-k$ are in $H$, and so the same is true of $i-k$. Thus, $i \sim_H k$, as required.
	
	The relation $\sim_H$ partitions $\Z_N$ into equivalence classes. From basic linear algebra results, it follows that $H$ is the subgroup generated by $\frac{N}{r}$, i.e., 
	$$ H=\left\langle \frac{N}{r}\right\rangle = \left\{ k\frac{N}{r}, \ k=0, 1, \dots, r-1\right\},$$ 
	and that all equivalence classes are in the form of $C_a=H+a$ for $a=0, 1, \dots, r-1$.
	
	For consistency of the notation, we denote the standard basis of $\C^N$ with ${\cal E}= \{e_0, e_1, \dots, e_{N-1}\}$ with the understanding that $e_j$ is the vector whose $j+1$ component is 1 and the other components are zero. Consider the new basis ${\cal E}'$, obtained from ${\cal E}$ by reordering the elements in the following way:
	\begin{align*}
		{\cal E}' = \Bigg\{ 
		& \underbrace{e_0, e_{r}, \dots, e_{(N/r-1)r}}_{\text{group 0}}, \, 
		\underbrace{e_1, e_{1+r}, \dots, e_{1+(N/r-1)r}}_{\text{group 1}}, \, \dots, \\
		& \dots, \, \underbrace{e_{r-1}, e_{2r-1}, \dots, e_{r-1+(N/r-1)r}}_{\text{group } r-1} 
		\Bigg\}.
	\end{align*}

	That is, the vectors in each group have indices in one of the equivalence classes. Let $P :\C^N \to \C^N$ be the linear application that maps the $j$-th vector of the basis ${\cal E}$ into the $j$-th vector of the ordered basis ${\cal E}'$. It is easy to verify that $P^{-1}=P^T$, and so the matrix of this application, that we still denote with $P$, is orthogonal.
	
	We now consider the matrix $B:=P^TAP$. Such matrix is obtained from $A$ by first rearranging the columns according to the groups, and then doing the same with the rows. We show that $B$ is block-diagonal. To do this, we show that if the indices $i, j$ belong to different equivalence classes, then $b_{i,j}=0$. Indeed, assume for the sake of contradiction that $b_{i,j} \ne 0$. Thus, the diagonal $D_{i-j}$ is nonzero, i.e., $i-j \in H$. We can conclude that $i \sim_H j$, which contradicts the assumption that $i$ and $j$ belong to different equivalence classes.
	
	We have proved that $B= \diag[B_0, \dots, B_{r-1}]$, where the $B_s$ are $\frac{N}{r} \times \frac{N}{r}$ blocks. The element $(i,j)$ of the block $B_s$ is $ b_{i,j}^{(s)} := a_{ir+s, jr+s}$. Note that if $A$ is a Toeplitz matrix, then $b_{i,j}^{(s)}= b_{i,j}^{(s')}$ for every $s, s' < \frac{N}{r}$ and so all the blocks are the same.
	
	\medskip
	Assume now that $U$ is not a subgroup of $\Z_N$. The relation $\sim_U$ defined in the first part of the proof is still reflexive and symmetric, but not necessarily transitive. To fix this problem, we define the "transitive closure" of $\sim_U$ as follows: 
	Two indices $i, j$ in the set $\{0,\dots,N-1\}$ are {\it chain-linked} in $U$, and we write $i \approx j$, if there exists a sequence $i=p_0, p_1, \dots, p_m=j$ in $\{0,\dots,N-1\}$ such that $p_0=i$, $p_m=j$ and $p_l - p_{l+1} \in U$ for all $0 \leq l < m$.
	
	It is easy to see that $\approx$ is an equivalence relation, which partitions $\Z_N$ into disjoint equivalent classes.
	
	Let $H = \langle U \rangle$ be the subgroup of $\mathbb{Z}_N$ generated by the indices of the nonzero diagonals, and let $C_0$ be the equivalence class containing zero. We show that $C_0=H$, and that $i \approx j$ if and only if $i \sim_H j$.
	
	Indeed, if $k \approx 0$, there exist a sequence $p_0, p_1, \dots, p_m$ such that $p_0=0$, $p_m=k$, and $p_l \sim_U p_{l+1}$ for $0 \leq l < m$. From $p_l \sim_U p_{l+1}$ we get that $p_{l+1}-p_l$ is the index of a nonzero diagonal, so $d_{l+1} := p_{l+1}-p_l \in U$. Thus,
\begin{align*}k &= p_m = (p_m-p_{m-1}) + (p_{m-1}-p_{m-2}) + \dots + (p_1-p_0)+p_0 \\ &= d_m+d_{m-1}+\dots+d_1.\end{align*}
	Each $d_j$ is in $U$, hence $k = d_1+d_2+\dots+d_m \in H$. This argument shows that $C_0 \subset H$; to prove that $H \subset C_0$, we observe that every $h \in H$ can be written as $h=d_1 + \dots + d_m$, where $d_j \in U$. If we let $p_0=0$, $p_1=d_1$, $p_2=d_1+d_2$, $p_3=d_1+d_2+d_3$, $\dots$, $p_m=d_1+d_2+d_3+\dots+d_m=h$, we can see at once that $0 \approx h$ and hence $h \in H$.
	
	\medskip
	The argument used to prove that $H=C_0$ can easily be modified to show that $i \approx j \iff i-j \in H$, that is, $i \approx j \iff i \sim_H j$.
	
	\medskip
	To conclude the proof  we need to show that $H$ is the subgroup generated by $\ell$, where $\ell = \gcd(N, \{d \in U\})$. 
	
	It is   well-known    that the subgroup of $\Z_N$ generated by a set of elements is exactly the subgroup generated by their greatest common divisor with $N$. Indeed, letting $U = \{h_1, \dots, h_m\}$, Bezout's identity guarantees the existence of integers $a_1, \dots, a_m$ and $b$ such that $a_1 h_1 + \dots + a_m h_m + b N = \ell$. Reducing modulo $N$ shows that $\ell \in \langle U \rangle = H$. Conversely, since $\ell$ divides each $h_i$, we have $h_i \in \langle \ell \rangle$ for all $i$, meaning $H \subseteq \langle \ell \rangle$. Thus, $H = \langle \ell \rangle$. 
	
	Because $\ell$ divides $N$, the cyclic subgroup $H$ has order $\frac{N}{\ell}$. As established in the first part of the proof, this implies that $A$ is orthogonally equivalent to a block-diagonal matrix with $\ell$ blocks, each of size $\frac{N}{\ell} \times \frac{N}{\ell}$, and provides the explicit expression for these blocks.
 
\noindent
{\bf Declaration of  Generative AI use}
We have  used ChatGPT to  improve  the flow of the introduction of the article, to proofread our main theorems and to search for references, which we have thoroughly verified for accuracy and relevance.

\end{document}